\documentclass{amsart}
\usepackage{graphicx}
\usepackage{amscd}
\usepackage{amsmath}
\usepackage{amsfonts}
\usepackage{amssymb}
\theoremstyle{plain}
\newtheorem{theorem}{Theorem}
\newtheorem{corollary}[theorem]{Corollary}

\newtheorem{proposition}[theorem]{Proposition}
\theoremstyle{definition}
\newtheorem{example}[theorem]{Example}

\newtheorem{definition}[theorem]{Definition}

\newtheorem{remark}[theorem]{Remark}
\theoremstyle{remark}

\begin{document}
\title{Lifting multiplicative lattices to ideal sytems}

\author{Tiberiu Dumitrescu, Mihai Epure and Alexandru Gica}

\address{Facultatea de Matematica si Informatica,University of Bucharest,14 A\-ca\-de\-mi\-ei Str., Bucharest, RO 010014,Romania}\email{tiberiu@fmi.unibuc.ro, tiberiu\_dumitrescu2003@yahoo.com}

\address{Simion Stoilow Institute of Mathematics of the Romanian AcademyResearch unit 5, P. O. Box 1-764, RO-014700 Bucharest, Romania}\email{mihai.epure@imar.ro, epuremihai@yahoo.com}

\address{Facultatea de Matematica si Informatica,University of Bucharest,14 A\-ca\-de\-mi\-ei Str., Bucharest, RO 010014,Romania}\email{alexandru.gica@unibuc.ro, alexgica@yahoo.com}


\begin{abstract}
\noindent 
We present a mechanism which lifts a multiplicative lattice to a (weak) ideal  system  on some monoid.
\end{abstract}

\maketitle
A {\em multiplicative lattice} is a complete lattice with least element $0$ and greatest element $1$, on which there is defined a commutative completely join distributive monoid operation  whose identity is $1$. We write simply {\em lattice} to mean a multiplicative lattice.
By a {\em monoid} we mean a commutative monoid with identity element $1$ and zero element $0$.

A (weak) ideal  system  on some monoid (see Definition \ref{4}) gives the multiplicative lattice of its $r$-ideals (see Theorem \ref{5}). In this  short paper we take the inverse direction providing a lifting procedure of multiplicative lattices  to  (weak) ideal  systems. This procedure is inspired by the work of Aubert \cite{Au} and Lediaev  \cite{L} where there are results on lifting   multiplicative lattices to x-systems.

We obtain the following results. Let $L$ be a lattice and $H$ a submonoid of $L$ generating $L$ as a lattice (such $H$ is named in this paper a {\em wire}, see Definition \ref{9}). Then $H$ gives   a weak ideal  system $r$ on $H$ (Theorem \ref{51}, Corollary \ref{83} and Proposition \ref{82}). 
This $r$ is an ideal  system iff $H$ is a so-called  {\em M-wire} (see Definition \ref{9}). 
A lattice which is liftable to an ideal system is generated by meet principal elements, while a  lattice domain which is generated by   principal elements is liftable to an ideal system (Proposition \ref{82}).
See the definition for "(meet) principal element" in the next paragraph.
In Proposition \ref{84} we investigate some M-wires of  the lattice $\mathbb{N}$ (with usual number  multiplication where $\bigvee=gcd$ and $\bigwedge=lcm$) given by the norm function of a ring of quadratic integers.
As an application of our results, we give a natural procedure to associate to a given lattice   $L$   another lattice   $L'$ generated by compact elements (see Remark \ref{10} and Example \ref{85}).
\\[2mm]

Let $L$ be a lattice. Denote by $\vee$ resp.  $\wedge$ its join resp. meet. If $a,b\in L$, we denote by $[a,b]$ the interval $\{x\in L| a\leq x\leq b\}$.
For $a,b\in L$, $(a:b)$ is the join of all $y\in L$ with $by\leq a$.
Recall the following definitions due to Dilworth \cite {D}.
An element $x\in L$ is said to be {\em meet principal} if $a\wedge xb = x((a:x) \wedge b)$ for all $a,b\in L$. Next $x$ is called 
{\em weak meet principal} if the preceding equality holds  for all $a\in L$ and $b=1$.
An element $x\in L$ is said to be {\em join principal} if 
$a\vee (b:x) = (ax \wedge b):x$ for all $a,b\in L$. Next $x$ is called {\em weak join principal} if the preceding equality holds  for all $a\in L$ and $b=0$. Finally $x\in L$ is called {\em (weak) principal} if it is both (weak) meet principal and (weak) join principal.

An element $x\in L$ is said to be {\em compact} if $x\leq \bigvee A$ with $A\subseteq L$ implies $x\leq \bigvee B$ for some finite subset $B$ of $A$. We say that a subset $C$ of $L$ {\em generates} $L$ if every element of $L$ is a join of some elements in $C$.
Any undefined notation or terminology is standard as in \cite{H} or \cite{A}.
\\[1mm]

We recall the definition of a  (weak) ideal system cf.
\cite[Chapter 2]{H}.

\begin{definition}\label{4}
Let $H$ be a  monoid. A {\em weak ideal system} on $H$ is a map $r:\mathcal{P}(H) \rightarrow \mathcal{P}(H)$ satisfying the following axioms:

 \quad\quad  $(s1)$ $XH\subseteq X_r$ for all $X\subseteq H$,
 
 \quad\quad   $(s2)$ $X\subseteq Y\subseteq H$ implies $X_r\subseteq Y_r$
 
 \quad\quad   $(s3)$ $(X_r)_r=X_r$ for all $X\subseteq H$,
 
 \quad\quad   $(s4)$ $cX_r\subseteq (cX)_r$
 for all $X\subseteq H$ and  $c\in H$.
 \\
A weak ideal system $r$ is called an {\em ideal system} if equality always holds  in $(s4)$.
Also a (weak) ideal system $r$ is said to be {\em finitary} if 

 \quad\quad    $(s5)$ $X_r=\bigcup \{ Z_r| Z $ finite subset of $X\}$ for all $X\subseteq H$.
\\
 The elements in the image of $r$ are called {\em $r$-ideals}.
\end{definition}

The next result  follows immediately from  \cite[Propositions 2.1 and 2.3]{H} and definitions.
 
\begin{theorem}\label{5}
Let $H$ be a  monoid and $r$ a weak  ideal system on $H$. 
Then the set
$$I_r(H) := \{ X_r \ |\ X\subseteq H\}$$   
of all r-ideals of $H$ 
is a  lattice w.r.t. following operations

\quad\quad   multiplication: $(X,Y)\mapsto (XY)_r$ for all $X,Y \in I_r(H)$,
 
\quad\quad join: $ \bigvee \Gamma := (\bigcup \Gamma)_r$ for all $\Gamma \subseteq I_r(H)$,

\quad\quad meet: $ \bigwedge \Gamma := \bigcap \Gamma$ for all $\Gamma \subseteq I_r(H)$,
\\
where $\bigcup \Gamma$ resp. $\bigcap \Gamma$ are the union resp.  intersection of all members of $\Gamma$.
\\ If $r$ is finitary, then   $S=:\{\{a\}_r\ |\ a\in H\}$ is a generationg submonoid of lattice $I_r(H)$ consisting of compact elements.
\end{theorem}

Let $L$ be a lattice.  We look for a weak ideal system   $r$ whose $r$-ideal lattice $I_r(H)$   is isomorphic to $L$.
In this case we say that $r$ is a  {\em lifting} of $L$ or that $L$ is {\em liftable} (to $r$). Getting inpiration from  \cite{Au} and \cite{L}, we introduce the following concept.

\begin{definition}\label{9}
Let $L$ be a lattice. By a {\em wire}
 $H\subseteq L$ we mean a  submonoid of  $L$ which generates $L$ as lattice. A wire $H$ is called an {\em M-wire} if it satisfies the  following condition:

$(M)$\ \ \ \ \  if $s\leq ta$ with $s,t\in H$ and $a\in L$, then $s=tu$ for some $u\in H\cap [0, a]$.
\end{definition}
 
The next theorem is the main result of the paper. 

\begin{theorem} \label{51}
Let $H$ be a wire of a lattice   $L$. Then  the map
\begin{center}
$r:\mathcal{P}(H) \rightarrow \mathcal{P}(H)$   given by  $X_r=H\cap [0,\bigvee X]$. 
\end{center}
 is a weak ideal system which is a lifting of $L$.
\end{theorem}
\begin{proof} 

We  check that $r$ satisfies conditions $(w1)$ to $(w4)$ of Definition \ref{4}.
Let $X\subseteq Y\subseteq   H$. 
 For $h\in H$ and $x\in X$,  we have $hx\leq x\leq \bigvee X$, so $hx\in X_r$, thus  $(w1)$ holds. 
Since $X\subseteq Y\subseteq H$, we have $\bigvee X\leq \bigvee Y$, so 
 $$X_r=H\cap [0,\bigvee X] \subseteq H\cap [0,\bigvee Y]=Y_r$$ thus  $(w2)$ holds. 
As $H$   generates   $L$, we have $$\bigvee X_r=\bigvee (H\cap [0,\bigvee X])=\bigvee X$$ so $$(X_r)_r=  H\cap [0,\bigvee X_r] = H\cap [0,\bigvee X] = X_r$$  thus  $(w3)$ holds. 
For $c\in H$ we have
$$
cX_r = c(H\cap [0,\bigvee X]) \subseteq H\cap [0,\bigvee cX] = (cX)_r
$$
so condition $(w4)$ 
holds.
We  show  that $L$ is isomorphic to the lattice of $r$-ideals $I_r(H)$.
Consider the maps 
$$f:I_r(H)\rightarrow L \ \ \mbox{ given by }  \ \ X\mapsto \bigvee X$$ and 
$$g:L\rightarrow I_r(H) \ \ \mbox{ given by } \ \ y\mapsto H \cap [0,y].$$
As $L$ is  generated by $H$, we have $\bigvee g(y)=y $, so 
 $g(y)$ is indeed an $r$-ideal of $H$.

For $X\in I_r(H)$, we have $$(gf)(X)=H\cap [0,\bigvee X]=X_r=X.$$ 
Also, for $y\in L$, we have 
$$(fg)(y)=\bigvee(H\cap [0,y])=y$$
as noticed above. Hence $f$ and $g$ are inverse to each other. For $X,Y\in W$, we have $$f((XY)_r)=\bigvee (XY)_r =\bigvee (XY) =(\bigvee X)(\bigvee Y)=f(X)f(Y)$$ so $f$ is a monoid morphism. 
If $X,Y\in I_r(H)$ and $X\subseteq Y$, then 
$$ f(X) = \bigvee X \leq \bigvee Y = f(Y).
$$
Conversely, if $x,y\in L$ and $x\leq y$, then 
$$ g(x) = H \cap [0,x] \subseteq H \cap [0,y] = g(y).
$$
Hence $f$ and $g$ are increasing maps.
Thus $f$ is a lattice isomorphism.
\end{proof}

\begin{corollary}\label{83}
Under  Theorem \ref{51} assumptions,  we have

$(i)$ $r$ is an ideal system iff   $H$ is an $M$-wire.

$(ii)$ $r$ is finitary iff   $H$ consists of compact elements.

\end{corollary}
\begin{proof}
$(i)$ implication $(\Leftarrow)$.
Let $h,c\in H$ such that $h\leq c(\bigvee X)$.
As $H$ is an M-wire, we get
$h=ck$ for some   $k\in H$ with $k\leq \bigvee X.$
Since $L$ is generated by $H$,
we have $$(cX)_r=H\cap [0,\bigvee cX] = H\cap [0,c(\bigvee X)] \subseteq c(H\cap [0,\bigvee X]) = cX_r$$
so   $(s4)$ holds. Thus $r$ is an ideal system. 

$(i)$ implication $(\Rightarrow)$. Suppose that $s\leq ta$ with $s,t\in H$ and $a\in L$. Since $H$ generates $L$, $a=\bigvee X$ for some $X\subseteq H$.  Then $s\in (tX)_r= tX_r$, so   $s=tu$ for some $u\in H$ with $u\leq a$.

$(ii)$ implication $(\Leftarrow)$.   Let $X\subseteq H$ and $a\in X_r$; so   $a\leq \bigvee X$. As $a$ is compact we get $a\leq \bigvee Z$ for some finite subset of $Z$ of $X$. Thus   $a\in Z_r$, so $r$ is finitary.

$(ii)$ implication $(\Rightarrow)$. Let $s\in H$ and $\{a_\alpha\}_{\alpha\in I} \subseteq L$ such that $s\leq \bigvee_{\alpha\in I} a_\alpha$. Write $a_\alpha = \bigvee X_\alpha$ with $X_\alpha\subseteq H$.
Then $s\in (\bigcup_{\alpha\in I} X_\alpha)_r$, so $s\in (\bigcup_{\alpha\in J} X_\alpha)_r$ for some finite subset $J\subseteq I$ since $r$ is finitary. We get $s\leq \bigvee_{\alpha\in J} a_\alpha$ so $s$ is compact.
\end{proof}

\begin{example}
Consider the lattice $L=\{0,1,a,b,c,d \}$ ordered by 
$a\leq b\leq d$ and  $a\leq c\leq d$
  with multiplication  
$$xy=0 \mbox{ for all }  x,y\in \{a,b,c,d\}.$$
It's   easy to check  that $H=\{0, a,b,c,1\}$ is a  wire, so $L$ lifts to a weak ideal system $r$ whose $r$-ideals are 
$$\{ 0\}, \{ 0,a\}, \{ 0,a,b\}, \{ 0,a,c\},\{ 0,a,b,c\}, \{ 0,a,b,c,1\}$$ 
cf. Theorem \ref{51}. As the weak meet elements are $0$, $a$ and $1$ is not liftable to an ideal system, cf. Proposition \ref{82} $(ii)$.
\end{example}

A lattice $L$ is called a {\em domain lattice} if $ab=0$ with $a,b\in L$ implies $a=0$ or $b=0$.

\begin{proposition}\label{82}
The following assertions are true.

$(i)$ Any lattice can be lifted to a weak ideal system.

$(ii)$ A lattice which is liftable to an ideal system is generated by meet principal elements.

$(iii)$ A  lattice domain which is generated by   principal elements is liftable to an ideal system.
\end{proposition}
\begin{proof}
Let $L$ be a lattice.  $(i)$ follows applying Theorem \ref{51} for 
 $H=L$. 
 
$(ii)$ Suppose that $L$   is liftable to an ideal system $r$ on a monoid $H$. Since the principal $r$-ideals $aH=\{a\}_r$, $a\in H$, generate $I_r(H)$ it suffices to show that each $aH$ is a   meet principal element of $I_r(H)$. Indeed, if $A,B$ are $r$-ideals, we have the obvious equality $A\cap Ba = a((A:a)\cap B)$. 

$(iii)$ Suppose that $L$ is generated by its subset $H$ of  principal elements. By \cite[Corollary 3.3]{D}, $H$ is a submonoid of $L$, so $H$ is a wire. Suppose that $s\leq ta$ with $s,t\in H$ and $a\in L$. As $t$ is principal, $s=tu$ for some $u\leq a$ in $H$ cf. \cite[Theorem 7]{AJ}.
So $H$ is an M-wire, hence $L$ is liftable to an ideal system cf. Theorem \ref{51}.
\end{proof}

\begin{example}\label{90}
Consider the lattice $\mathbb{N}$ with usual number  multiplication where $\bigvee=gcd$ and $\bigwedge=lcm$.
Let $D$ be a ring of algebraic integers. Sending each 
$X\subseteq D$ into $X_r=$ the ideal generated by $X$, we get an ideal system whose     $r$-ideal    lattice is the usual ideal lattice $I_D$ of $D$. 
Note that the set of principal ideals of $D$ is an  M-wire of $I_D$. 
It is well-known that $I_D$ is isomorphic to $\mathbb{N}$. 
So lattice $\mathbb{N}$ can be lifted to an ideal system in infinitely many ways and it has infinitely many M-wires. Our next result explores some M-wires of $\mathbb{N}$ given by the norm function on a ring of quadratic integers.
\end{example}

Let $D$ be a nonfactorial ring of quadratic integers with class group $G$
and let $N:D\rightarrow \mathbb{N}$ the absolute value of its norm function. Let $S$ be the multiplicatively closed subset of $\mathbb{N}$ generated by the image $Im(N)$ of $N$ and the set $I$ of all prime numbers which are inert in $D$.

\begin{proposition} \label{84}
With notation above, $S$ is an M-wire on lattice $\mathbb{N}$ (see Example \ref{90}) iff $G$ is a finite product of copies of $\mathbb{Z}_2$.
\end{proposition}
\begin{proof}

 We shall use repeatedly the following well-known Number Theory facts: $D$ is a Dedekind domain, $G$ is finite and every class $g\in G$ contains infinitely many prime ideals of $D$.
We first prove that $S$ generates $\mathbb{N}$ as a lattice (i.e. $S$ is a wire).
 It suffices to show that every prime number $p\in \mathbb{N}-I$ is the gcd of some numbers in $S$. Take a prime ideal $P$ of norm $p$. If $P$ is principal, then $p\in S$. Suppose that $P$ is not principal and let $e$ be its class in $G$. Inside $-e$ take another two prime ideals $Q$ and $R$ of norms $q$ and $r$ respectively. We may arrange that $p,q,r$ are distinct.
As $PQ$ and $PR$ are principal ideals, we get $pq,pr\in S$ and $p$ is their gcd.

Therefore, we may assume from the very beginning that $S$ is a wire. It's easy to see that $S$ is an $M$-wire iff $S$ is closed under division iff $Im(N)$ is closed under division (i.e. if $a,b\in Im(N)-\{0\}$ and $a|b$, then $b/a\in Im(N)$). So it remains to show that $Im(N)$ is closed under division iff $G$ is a finite product of copies of $\mathbb{Z}_2$.

Suppose that $Im(N)$ is closed under division. Then $G$ has no odd order element. Deny.  Let $P$ be a prime ideal of $D$ whose ideal class has  odd order $m$ and let $p$ be the norm of $P$. Then $p^m\in Im(N)$ and, since $p^2$ is clearly in $Im(N)$, we get that $p\in Im(N)$, as $Im(N)$ is closed under division. But this is a contradiction because $P$ is not principal.
To show that $G$ is a finite product of copies of $\mathbb{Z}_2$, it suffices to prove that $G$ has no element of order four. Deny. Let $g\in G$ of order four. Select  prime ideals $P,Q$  of norms $p,q$ in classes $g,2g$ respectively. Denote the conjugate of $P$ by $\overline{P}$. Since $P\overline{P}$ and $P^2Q$ are principal ideals, we get that $p^2$ and $p^2q$ are in $Im(N)$, so $q\in Im(N)$, as $Im(N)$ is closed under division. But this is a contradiction because $Q$ is not principal.

Conversely, suppose that $G$ is a finite product of copies of $\mathbb{Z}_2$. Let $a,b\in D-\{0\}$ such that $N(a)\ |\ N(b)$. Since $D$ is a Dedekind domain, we may consider the prime power factorizations  $aD=P_1\cdots P_n$ and $bD=P_{n+1}\cdots P_m$. We use now the fact that each element in $G$ has order $\leq 2$. 
We replace some of the factors $P_i$ by their conjugate such that finally no pair of distinct conjugates appears in list $\{ P_1,...,P_m\}$. 
Doing this we change $a$ and $b$ but we preserve their norm. Moreover, in the new setup it follows that $a$ divides $b$, so $N(b)/N(a)=N(b/a)\in Im(N)$.
\end{proof}

\begin{remark}
With notation above, $S$ is an M-wire on lattice $\mathbb{N}$ provided $D=\mathbb{Z}[\sqrt{-5}]$ (since its class group is $\mathbb{Z}_2$) but $S$ is not an M-wire on lattice $\mathbb{N}$ if $D=\mathbb{Z}[\sqrt{-17}]$ since its class group is $\mathbb{Z}_4$ or $N(5+\sqrt{-17})=42$, $N(2+\sqrt{-17})=21$ but there is not a single element in $\mathbb{Z}[\sqrt{-17}]$ of norm 2.
\end{remark}

We put Theorem \ref{51} to work. Let $r$ be a weak ideal system on a monoid $H$. Recall that the map $r_s:\mathcal{P}(H) \rightarrow \mathcal{P}(H)$ 
given by
$$ X_{r_s}=\bigcup \{ Z_r| Z \mbox{ finite subset of } X\} \mbox{ for all } X\subseteq H$$ is a finitary weak ideal system 
called the finitary weak ideal system associated to $r$.
See \cite[Chapter 3]{H} for details.

\begin{remark}\label{10}
We give the following application of Theorem \ref{51}.
To a lattice   $L$   we can canonically associate a lattice   $L'$ generated by compact elements. 
Let $r$ be the ideal system on $H$ constructed in  Theorem \ref{51} for $H=L$. Let $r_s$ be the finitary weak ideal system associated to $r$ recalled above.
By Theorem \ref{5}, the lattice  $L'=I_{r_s}(H)$ of all $r_s$-ideals of $H$ 
is generated by compact elements. By Theorem \ref{51} and the definition of $r_s$ we get
$$X_{r_s}=
\{ h\in L\ |\ h\leq h_1\vee ... \vee h_n \mbox{ for some } h_1,...,h_n\in X\}.
$$
We get the set embedding
$$ L\rightarrow L',\ \ x \mapsto [0,x]
$$
which is a lattice isomorphism when $L$ is generated by compact elements, because in that case $r=r_s$. 
\end{remark}

\begin{example}\label{85}
As an illustration of Remark \ref{10} consider the lattice $L=[0,1]\subseteq \mathbb{R}$ with usual number multiplication with $\bigvee=sup$ and $\bigwedge=inf$. No nonzero element $x$ of $L$ is  compact because 
$$x = \bigvee \{ x-1/n\ |\ n\geq 1/x,\ n\in \mathbb{N}\}$$
but any finite subjoin is $<x$.
Performing the construction in Remark \ref{10} we get the lattice
$$L'=A\cup B \mbox{ with } A=\{\ [0,x]\ |\ x\in [0,1]\} \mbox{ and } B=\{\ [0,x)\ |\ x\in [0,1]\}.
$$
The multiplication in $L'$ is the usual interval multiplication.
For $X\subseteq L$ with $a=sup(X)$, we get that $X_{r_s}$ is $[0,a]$ resp. $[0,a)$ if $a\in X$ resp. $a\notin X$ .
Each element of $A$ is compact in $L'$.
\end{example}


\begin{thebibliography}{11111}

\bibitem{A} D.D. Anderson, Abstract commutative ideal theory without chain conditions, Algebra Univ. 6 (1976), 131-145.

\bibitem{AJ} D.D. Anderson and  E.W. Johnson, 
Dilworth's principal elements, 
Algebra Universalis, 36 (1996), 392-404.
 
\bibitem{Au} K.E. Aubert,
Un theoreme de representation dans la theorie des ideaux,
C. R. Acad. Sci., Paris 242 (1956), 320-322.

\bibitem{D} R. Dilworth, 
Abstract commutative ideal theory,
Pacific J. Math. 12 (1962), 481-498. 


 \bibitem{H} F. Halter-Koch, {\em Ideal Systems: an Introduction to Multiplicative Ideal Theory}, Marcel Dekker, New York, 1998.

\bibitem{L} J. Lediaev, 
Relationship between Noether lattices and x-systems,
Acta Math. Acad. Sci. Hung. 21 (1970), 323-325. 
 
\end{thebibliography}
\end{document}